\documentclass[10pt]{amsart}
\usepackage{amsfonts,amssymb,amscd,amsmath,enumerate,verbatim,calc}

\setbox0=\hbox{$+$}
\newdimen\plusheight
\plusheight=\ht0
\def\+{\;\lower\plusheight\hbox{$+$}\;}

\setbox0=\hbox{$-$}
\newdimen\minusheight
\minusheight=\ht0
\def\-{\;\lower\minusheight\hbox{$-$}\;}

\setbox0=\hbox{$\cdots$}
\newdimen\cdotsheight
\cdotsheight=\plusheight%\ht0
\def\cds{\lower\cdotsheight\hbox{$\cdots$}}

\setbox0=\hbox{$+$}
\numberwithin{equation}{section}
 \theoremstyle{plain}
\newtheorem{thm}{Theorem}[section]
\newtheorem{prob}[thm]{Problem}

\begin{document}
\title[Prime characterization from binomial coefficient] {A characterization of a prime $p$ from the binomial coefficient ${n \choose p}$ }
\author{Alexandre Laugier}
\address{Lyc{\'e}e professionnel hotelier La Closerie, 10 rue Pierre Loti - BP 4, 22410 Saint-Quay-Portrieux, France}
\email{laugier.alexandre@orange.fr}
\author{Manjil P.~Saikia}
\thanks{The second author is supported by a DST INSPIRE Scholarship 422/2009 from the Department of Science and Technology, Government of India.}
\address{Department of Mathematical Sciences, Tezpur University, Napaam, Sonitpur, Assam, Pin-784028, India}
\email{manjil@gonitsora.com}

\maketitle

%\vspace*{0.5in}
%\begin{center}
%{\bf A CHARACTERIZATION OF A PRIME $p$ FROM THE BINOMIAL COEFFICIENT ${n\choose p}$ }\\[5mm]
%{\footnotesize  ALEXANDRE LAUGIER and MANJIL P.~SAIKIA\footnote{Corresponding Author: manjil@gonitsora.com}}\\[3mm]
%\end{center}

\begin{abstract}
 We complete a proof of a theorem that was inspired by an Indian Olympiad problem, which gives an interesting characterization of a prime number $p$ with respect to the binomial coefficients ${n\choose p}$. We also derive a related result which generalizes the theorem in one direction.
\end{abstract}

\vskip 3mm


\vskip 3mm

\noindent{\footnotesize 2010 Mathematical Reviews Classification
Numbers: 11A07, 11A41, 11A51, 11B50, 11B65, 11B75.}

\section{{Introduction and Motivation}}
\begin{prob}
$7$ divides $\binom{n}{7}-\lfloor\frac{n}{7}\rfloor$, $\forall n \in \mathbb{N}$.
\end{prob}

The above appeared as a problem in the Regional Mathematical Olympiad, India in 2003. Later in 2007, a similar type of problem was set in the undergraduate admission test of Chennai Mathematical Institute, a premier research institute of India where $7$ was replaced by $3$.

This became the basis of the following

\begin{thm}[\cite{mps}, Saikia-Vogrinc]\label{mps-jv}
A natural number $p>1$ is a prime if and only if $\binom{n}{p}-\lfloor\frac{n}{p}\rfloor$ is divisible by $p$ for every non-negative $n$, where $n>p+1$ and the symbols have their usual meanings.
\end{thm}

\section{{Proof of Theorem \ref{mps-jv}}}

In \cite{mps}, the above theorem is proved. The authors give three different proofs, however the third proof is incomplete. We present below a completed version of that proof.

\begin{proof}
 First we assume that $p$ is prime. Now we consider $n$ as $n=ap+b$ where $a$ is a non-negative integer and $b$ an integer $0\leq b<p$. Obviously,
\begin{equation}
\left\lfloor \frac{n}{p}\right\rfloor=\lfloor \frac{ap+b}{p}\rfloor\equiv a~(mod~p).
\end{equation}
Now let us calculate $\binom{n}{p}~(mod~p)$.
\begin{eqnarray*}
\binom{n}{p}&=&\binom{ap+b}{p}\\
&=&\frac{(ap+b)\cdot(ap+b-1)\cdots(ap+1)\cdot ap\cdot(ap-1)\cdots(ap+b-p+1)}{p\cdot(p-1)\cdots 2\cdot1}\\
&=&\frac{a\cdot(ap+b)\cdot(ap+b-1)\cdots(ap+1)\cdot(ap-1)\cdots(ap+b-p+1)}{(p-1)\cdot(p-2)\cdots 2\cdot 1}\\
&=&\frac{a X}{(p-1)!}
\end{eqnarray*}
where $X=(ap+b)\cdot(ap+b-1)\cdots(ap+1)\cdot(ap-1)\cdots(ap+b-p+1)$.

We observe that there are $(p-1)$ terms in $X$ and each of them has one of the following forms,\\ (a) $ap+r_1$, or\\ (b) $ap-r_2$\\ where $1\leq r_1\leq b$ and $1\leq r_2\leq (p-1-b)$.\\ Thus any two terms from either $(a)$ or $(b)$ differs by a number strictly less than $p$ and hence not congruent modulo $p$. Similarly, if we take two numbers - one from $(a)$ and the other from $(b)$, it is easily seen that the difference between the two would be $r_1+r_2$ which is at most $(p-1)$ (by the bounds for $r_1$ and $r_2$); thus in this case too we find that the two numbers are not congruent modulo $p$. Thus the terms in $X$ forms a reduced residue system modulo $p$ and so, we have,
\begin{eqnarray}
X\equiv(p-1)!\;(mod\;p)
\end{eqnarray}
Thus using $(2.2)$ we obtain,
\begin{eqnarray}
\binom{n}{p}=a\frac{X}{(p-1)!}\equiv a\;(mod\;p)
\end{eqnarray}
So, $(2.1)$ and $(2.3)$ combined gives
\begin{eqnarray}
\left\lfloor \frac{n}{p}\right\rfloor\equiv \binom{n}{p}\;(mod\;p).
\end{eqnarray}
So, forward implication is proved.

To prove the reverse implication, we adopt a contrapositive argument meaning that if $p$ were not prime (that is composite) then we must construct an $n$ such that $(4)$ does not hold. So, let $q$ be a prime factor of $p$. We write $p$ as $p=q^x k$, where $(q,k)=1$. In other words, $x$ is the largest power of $q$ such that $q^x|p$ but $q^{x+1}\!\not|\,p$ (in notation, $q^x || p$). By taking, $n=p+q=q^xk+q$, we have
\begin{eqnarray*}
\binom{p+q}{p}=\binom{p+q}{q}=\frac{(q^xk+q)(q^xk+q-1)\dots (q^xk+1)}{q!}
\end{eqnarray*}
which after simplifying the fraction equals $(q^{x-1}k+1)\frac{(q^xk+q-1)\dots (q^xk+1)}{(q-1)!}$. Clearly, $(q^xk+q-1)\dots (q^xk+1)\equiv (q-1)!\not\equiv 0\;(mod\;q^x)$. Therefore,
\begin{eqnarray*}
\frac{(q^xk+q-1)\dots (q^xk+1)}{(q-1)!}\equiv 1~(mod~q^x)
\end{eqnarray*}
 and
 \begin{eqnarray*}
\binom{p+q}{p}\equiv q^{x-1}k+1~(mod~q^x).
 \end{eqnarray*}
On the other hand obviously,
\begin{eqnarray*}
\left\lfloor\frac{p+q}{p}\right\rfloor=\left\lfloor\frac{q^xk+q}{q^xk}\right\rfloor\equiv 1~(mod~q^x).
\end{eqnarray*}
Now, since $(q,k)=1$, it follows that $q^{x-1}k+1\not\equiv 1\;(mod\;q^x)$. So we conclude,
\begin{eqnarray}
\binom{p+q}{p}\not\equiv \left\lfloor\frac{p+q}{p}\right\rfloor\;(mod\;q^x).
\end{eqnarray}
So, $p\nmid (\binom{p+q}{p}-\lfloor\frac{p+q}{p}\rfloor)$, for if $p|(\binom{p+q}{p}-\lfloor\frac{p+q}{p}\rfloor)$, then since $q^{x}|p$, we would have $q^x|(\binom{p+q}{p}-\lfloor\frac{p+q}{p}\rfloor)$, a contradiction to  $(5)$. Thus, $\binom{p+q}{p}\not\equiv\lfloor\frac{p+q}{p}\rfloor\;(mod\;p).$ Hence we are through with the reverse implication too.

This completes the proof of Theorem \ref{mps-jv}.
\end{proof}

\section{{Another simple result}}

We state and prove the following simple result which generalizes one part of Theorem \ref{mps-jv}

\begin{thm}\label{gen}
For $n=ap+b=a_{(k)}p^k+b_{(k)}$, we have  $$
{a_{(k)}p^k+b_{(k)}\choose p^k}
-\left\lfloor\frac{a_{(k)}p^k+b_{(k)}}{p^k}\right\rfloor\equiv
0~(\textup{\textup{mod}}~p)
$$
with $p$ a prime, $0\leq b_{(k)}\leq p^k-1$ and $k$ a positive integer
such that $1\leq k\leq l$, where $$n=a_0+a_1p+\ldots+a_kp^k+a_{k+1}p^{k+1}+\ldots+a_lp^l$$ and
for $k\geq 1$
$$
a_{(k)}=a_k+a_{k+1}p+\ldots+a_lp^{l-k}
$$
and
$$
b_{(k)}=a_0+a_1p+\ldots+a_{k-1}p^{k-1}.
$$.
\end{thm}

The proof of this follows from the reasoning of the proof of Theorem \ref{mps-jv} although there are some subtleties.

In particular, we have
$$
a=a_{(1)}=a_1+a_{2}p+\ldots+a_lp^{l-1}
$$
and
$$
b=b_{(0)}=a_0.
$$
For $k=0$, we set the convention that $a_{(0)}=n=a_0+a_{1}p+\ldots+a_lp^{l}
$ and $b_{(0)}=0$.
Notice that Theorem \ref{gen} is obviously true for $k=0$. But the
case $k=0$ doesn't correspond really to a power of $p$ where $p$ is a
prime.

\begin{proof}
We have
$$\binom{n}{p^k}=\binom{a_{(k)}p^k+b_{(k)}}{p^k}$$
$$=\frac{(a_{(k)}p^k+b_{(k)})\cdot(a_{(k)}p^k+b_{(k)}-1)\cdots(a_{(k)}p^k+1)\cdot
a_{(k)}p^k\cdot(a_{(k)}p^k-1)\cdots(a_{(k)}p^k+b_{(k)}-p^k+1)}
{p^k\cdot(p^k-1)\cdots 2\cdot1}$$
$$=\frac{a_{(k)}\cdot(a_{(k)}p^k+b_{(k)})\cdot(a_{(k)}p^k+b_{(k)}-1)\cdots
(a_{(k)}p^k+1)\cdot(a_{(k)}p^k-1)\cdots(a_{(k)}p^k+b_{(k)}-p^k+1)}
{(p^k-1)\cdot(p^k-2)\cdots 2\cdot 1}.$$ Thus we obtain
$$
(p^k-1)!\binom{n}{p^k}=a_{(k)}\left({\displaystyle\prod^b_{r=1}(a_{(k)}p^k+r)}
\right)\,
\left({\displaystyle\prod^{p^k-1-b}_{r=1}(a_{(k)}p^k-r)}\right).
$$ Or $a_{(k)}p^k+r\equiv r~(\textup{\textup{mod}}~p^k)$ and
$a_{(k)}p^k-r\equiv -r\equiv p^k-r~(\textup{\textup{mod}}~p^k)$ with
$0<r<p^k$. It follows
$$
\left({\displaystyle\prod^b_{r=1}(a_{(k)}p^k+r)}
\right)\,
\left({\displaystyle\prod^{p^k-1-b}_{r=1}(a_{(k)}p^k-r)}\right)
\equiv\left({\displaystyle\prod^b_{r=1}r}\right)\,
\left({\displaystyle\prod^{p^k-1-b}_{r=1}(p^k-r)}
\right)~(\textup{\textup{mod}}~p^k).
$$ Since
$$
\left({\displaystyle\prod^b_{r=1}r}\right)\,
\left({\displaystyle\prod^{p^k-1-b}_{r=1}(p^k-r)}\right)
=\left({\displaystyle\prod^b_{r=1}r}\right)\,
\left({\displaystyle\prod^{p^k-1}_{r=b+1}r}\right)
={\displaystyle\prod^{p^k-1}_{r=1}r}=(p^k-1)!
$$ we have
$$
\left({\displaystyle\prod^b_{r=1}(a_{(k)}p^k+r)}
\right)\,
\left({\displaystyle\prod^{p^k-1-b}_{r=1}(a_{(k)}p^k-r)}\right)
\equiv(p^k-1)!~(\textup{\textup{mod}}~p^k).
$$

We can notice that,
$$
(p^k-1)!=q(p-1)!\,p^{1+p+\ldots+p^{k-1}-k}
$$
with $\textup{gcd}(p,q)=1$ and because $\textup{ord}_{p}((p^k-1)!)=1+p+\ldots+p^{k-1}-k$.
Therefore we have
$$
a_{(k)}c_{(k)}p^{k(p^k-1)}+(p^k-1)!\left\{a_{(k)}-\binom{n}{p^k}\right\}=0.
$$
Equivalently
$$
a_{(k)}c_{(k)}p^{k(p-1)(1+p+\ldots+p^{k-1})}+q(p-1)!\,p^{1+p+\ldots+p^{k-1}-k}
\left\{a_{(k)}-\binom{n}{p^k}\right\}=0
$$

Dividing the above equation by $p^{1+p+\ldots+p^{k-1}-k}$ we have
$$
q(p-1)!\left\{a_{(k)}-\binom{n}{p^k}\right\}
+a_{(k)}c_{(k)}p^{k+(k(p-1)-1)(1+p+\ldots+p^{k-1})}=0.
$$
Thus
$$
q(p-1)!\left\{a_{(k)}-\binom{n}{p^k}\right\}\equiv 0~(\textup{\textup{mod}}~p^k)
$$
Since if $m\equiv
n~(\textup{\textup{mod}}~p^k)$ implies $m\equiv
n~(\textup{\textup{mod}}~p)$ (the converse is not always
true), we also have
$$
q(p-1)!\left\{a_{(k)}-\binom{n}{p^k}\right\}\equiv 0~(\textup{\textup{mod}}~p).
$$
As $q(p-1)!$ with $\textup{gcd}(p,q)=1$ and $p$ are relatively prime, we get
$$
\binom{n}{p^k}-a_{(k)}\equiv 0~(\textup{\textup{mod}}~p).
$$

We finally have
$$
\binom{n}{p^k}\equiv
\left\lfloor\frac{n}{p^k}\right\rfloor~(\textup{\textup{mod}}~p).
$$
\end{proof}
 
\begin{thm}\label{t3.2}
Let $p$ be a prime number, let $k$ be a natural number and 
let $x$ be a positive integer such that
$$
x\equiv r\pmod {p^k}
$$
with $0\leq r<p^k$. Denoting $q=\lfloor\frac{x}{p^k}\rfloor$ the
quotient of the division of $x$ by $p^k$, if
there exists $s\in\mathbb{N}^{\star}$ for which
$$
\lfloor\frac{x}{p^{ks}}\rfloor=q^s
$$
then we have
$$
x\equiv r\pmod {p^{ks}}.
$$
\end{thm}

\begin{proof}
Given $p$ a prime number, let $x$ be a positive integer such that
$x\equiv r\pmod {p^k}$ with $0\leq r<p^k$. Denoting
$q=\lfloor\frac{x}{p^k}\rfloor$, we assume that
there exists $s\in\mathbb{N}^{\star}$ for which
$$
\lfloor\frac{x}{p^{ks}}\rfloor=q^s.
$$
If $k=0$, the result is obvious since for all integers $x,r$, we have
$x\equiv r\pmod 1$. In the following, we assume that $k\in\mathbb{N}^{\star}$.
\\[0.1in]
Then, we have
$$
x=qp^k+r
$$
and
$$
x=q^sp^{ks}+r'
$$
with $0\leq r'<p^{ks}$. It comes that
$$
qp^k+r=q^sp^{ks}+r'.
$$
So ($s\in\mathbb{N}^{\star}$)
$$
q^sp^{ks}-qp^k=r-r'\geq 0.
$$
Since $0\leq r<p^k$, we have
$$
0\leq r-r'<p^k.
$$
Moreover, rewriting the equality $q^sp^{ks}-qp^k=r-r'$ as 
$$
qp^k(q^{s-1}p^{k(s-1)}-1)=r-r'
$$
we can notice that $p^k|r-r'$. Since $0\leq r-r'<p^k$, it is only possible if
$r-r'=0$ and so
$$
r=r'.
$$
From the equality $x=q^sp^{ks}+r'$, we deduce that
$$
x\equiv r\pmod {p^{ks}}.
$$
\end{proof}

A consequence of the Theorem \ref{t3.2} is that if an integer $y$ is
congruent to a positive integer $x$ modulo $p^k$ such that $x\equiv
r\pmod {p^k}$, provided the
conditions stated in the Theorem \ref{t3.2} are fulfilled, we have also
$y\equiv r\pmod {p^{ks}}$.

It can be verified easily that the product 
$\left(\prod^b_{r=1}(a_{(k)}p^k+r)\right)\,
\left(\prod^{p^k-1-b}_{r=1}(a_{(k)}p^k-r)\right)$ contains the term 
$(a_{(k)}p^k)^{p^k-1}=a_{(k)}^{p^k-1}p^{k(p^k-1)}$. The term
$(a_{(k)}p^k)^{p^k-1}$ is the only
term in $p^{k(p^k-1)}$ which appears in the decomposition of this product into
sum of linear combination of powers of $p$. Notice also that the number
$k(p^k-1)$ is the greatest exponent 
of $p$ in this product 
when we decompose this product into sum of linear combination
of powers of $p$ (like a polynomial expression in variable
$p$). Afterwards, we write  
$a_{(k)}^{p^k-1}$ as $c_{(k)}$ in order to simplify the notation. Thus, 
the quotient of the division of this product by
$p^{k(p^k-1)}$ is $c_{(k)}=a_{(k)}^{p^k-1}$.

So, from the Theorem \ref{t3.2}, we can now write
$$
\left({\displaystyle\prod^b_{r=1}(a_{(k)}p^k+r)}
\right)\,
\left({\displaystyle\prod^{p^k-1-b}_{r=1}(a_{(k)}p^k-r)}\right)
=c_{(k)}p^{k(p^k-1)}+(p^k-1)!.
$$

\section{{Acknowledgements}}
The authors are grateful to Professor Nayandeep Deka Baruah for going through an earlier version of this work and offering various helpful suggestions which
helped to make the presentation much clearer. The authors are also thankful to Ankush Goswami for cleaning up the exposition in the proof of Theorem \ref{mps-jv}. The second named author would also like to thank Bishal Deb for pointing out to him a factual error in \cite{mps} and also to Parama Dutta for a careful reading of the first two sections of this paper.


\begin{thebibliography}{99}
\bibitem{dmb}D.~M.~Burton, \emph{Elementary Number Theory}, 6 ed., Tata McGraw--Hill, 2010.
\bibitem{el} E.~Lucas, \emph{Th$\acute{e}$orie des Fonctions Num$\acute{e}$riques Simplement P$\acute{e}$riodiques}, Amer.~J.~Math., \textbf{1} (2), 184--196; \textbf{1} (3), 197--240; \textbf{1} (4), 289--321 (1878).
\bibitem{mps}M.~P.~Saikia, J.~Vogrinc, \emph{A Simple Number Theoretic Result}, J.~Assam Acad. Math., \textbf{3}, Dec., 91-96, 2010 ({\tt arXiv:1207.6707 [math.NT]}).
\bibitem{mps2}M.~P.~Saikia, J.~Vogrinc, \emph{Binomial Symbols and Prime Moduli}, J.~Ind.~Math.~Soc., \textbf{79}, (1-4), 137-143, 2011.

\end{thebibliography}
\end{document}